\documentclass[leqno,12pt]{article}

\def\today{${\scriptscriptstyle\number\day-\number\month-\number\year}$}
\usepackage{graphicx}
\usepackage{fancyhdr}
\usepackage{amsmath}
\usepackage{amsthm}
\usepackage{amsfonts}
\newtheorem{theorem}{Theorem}[section]
\newtheorem*{thA}{Theorem A}
\newtheorem*{thB}{Theorem B}

\newtheorem{lemma}[theorem]{Lemma}

\newtheorem*{assA}{Assumption A}
\theoremstyle{definition}
\newtheorem*{reC}{Remark C}

\newtheorem{remark}[theorem]{Remark}

\def\address#1{{\center{#1}}}
\date{}
\def\m@th{\mathsurround=0pt}
\def\eqal#1{\null\,\vcenter{\openup\jot\m@th
 \ialign{\strut\hfil$\displaystyle{##}$&&$\displaystyle{{}##}$\hfil
 \crcr#1\crcr}}\,}
\def\matrix#1{\null\,\vcenter{\normalbaselines\m@th
 \ialign{\hfil$##$\hfil&&\quad\hfil$##$\hfil\crcr
 \mathstrut\crcr\noalign{\kern-\baselineskip}
 #1\crcr\mathstrut\crcr\noalign{\kern-\baselineskip}}}\,}
\def\N{{\Bbb N}}
\def\R{{\Bbb R}}
\def\divv{{\rm div}\,}
\def\rot{{\rm rot}\,}
\def\const{{\rm const}}
\def\esssup{\mathop{\rm esssup}\limits}

\numberwithin{equation}{section}

\title{On regular periodic solutions to the Navier-Stokes equations}
\author{Wojciech M. Zaj\c{a}czkowski}

\begin{document}
\input amssym.def
\input amssym.tex
\maketitle
%%%%%%%%%%%%%%%%%%%%%%%%%%%%%%%%%%%%%%%%%%%%%%%%%%%%%%%%%%%%%%%%%%%%%%%%%
\thispagestyle{fancy}
%%%%%%%%%%%%%%%%%%%%%%%%%%%%%%%%%%%%%%%%%%%%%%%%%%%%%%%%%%%%%%%%%%%%%%%%%

\address{Institute of Mathematics, Polish Academy of Sciences,\\
\'Sniadeckich 8, 00-656 Warsaw, Poland\\
e-mail:wz@impan.gov.pl\\
Institute of Mathematics and Cryptology, %\\
Cybernetics Faculty, \\
Military University of Technology,\\% Warsaw\\
S. Kaliskiego 2, 00-908 Warsaw, Poland\\}

\begin{abstract}
We find a global a priori estimate for solutions to the Navier-Stokes equations with periodic boundary conditions guaranteeing in view of the Serrin type condition the existence of global regular solutions. We derive the following estimate
$$
\|V(t)\|_{H^1(\Omega)}\le c,
\eqno(1)
$$
where $V$ is the velocity of the fluid.

\noindent
The estimate (1) is proved in two steps. First we derive a global estimate guaranteeing the existence of global regular solutions to weakly compressible Navier-Stokes equations with large second viscosity, density close to a constant and gradient part of velocity small.\\
Next we show that solutions to the Navier-Stokes equations remain close to solutions to the weakly compressible Navier-Stokes equations if the corresponding initial data and external forces are sufficiently close.
\end{abstract}


\noindent
{\bf Key Words and Pharases}: Navier-Stokes equations, global existence of regular solutions, weakly compressible Navier-Stokes equations, stability, the Serrin type condition

\section{Introduction}\label{s1}

We are going to prove existence of global regular periodic solutions to the Navier-Stokes equations in a box $\Omega\subset\R^3$
\begin{equation}\eqal{
&a(V_t+V\cdot\nabla V)-\mu\Delta V+\nabla P=aF\quad &{\rm in}\ \ \Omega\times\R_+,\cr
&\divv V=0\quad &{\rm in}\ \ \Omega\times\R_+,\cr
&V|_{t=0}=V_0\quad &{\rm in}\ \ \Omega,\cr}
\label{1.1}
\end{equation}
where $V=(V_1(x,t),V_2(x,t),V_3(x,t))\in\R^3$ is the velocity of the fluid, $x=(x_1,x_2,x_3)$ are the Cartesian coordinates, $P=P(x,t)\in\R$ is the pressure, $\mu>0$ is the viscosity coefficient, $F=F(x,t)=(F_1(x,t),F_2(x,t),F_3(x,t))\in\R^3$ is the external force and $a$ is positive constant.

Since the existence of global regular solutions to weakly compressible (second viscosity coefficient large, density close to a constant and divergence of velocity small) Navier-Stokes equations is known (see \cite{Z1}) we are looking for solutions to (\ref{1.1}) as for stability of these regular solutions. Therefore, the weakly compressible barotropic motions are described by the following problem
\begin{equation}\eqal{
&\varrho v_t+\varrho v\cdot\nabla v-\mu\Delta v-\nu\nabla\divv v+\nabla p=\varrho f\quad &{\rm in}\ \ \Omega\times\R_+,\cr
&\varrho_t+\divv(\varrho v)=0\quad &{\rm in}\ \ \Omega\times\R_+,\cr
&v|_{t=0}=v_0,\ \ \varrho|_{t=0}=\varrho_0\quad &{\rm in}\ \ \Omega,\cr}
\label{1.2}
\end{equation}
where $\varrho=\varrho(x,t)\in\R_+$ is the density of the fluid, $v=(v_1(x,t),v_2(x,t)$, $v_3(x,t))\in\R^3$ is velocity, $p=A\varrho^\varkappa$, $\varkappa>1$, $A$- constant, is the pressure.

\noindent
By weakly compressible motions we mean such motions that
\begin{equation}
\varrho=a+\eta,
\label{1.3}
\end{equation}
where $a$ is the constant from (\ref{1.1}) and $\eta$ is small. Moreover, $\nu$ is large and $\divv v(0)$ is as small as we want. Then problem (\ref{1.2}) takes the form
\begin{equation}\eqal{
&(a+\eta)(v_t+v\cdot\nabla v)-\mu\Delta v-\nu\nabla\divv v+a_0\nabla\eta\cr
&=(p_\varrho(a)-p_\varrho(a+\eta))\nabla\eta+(a+\eta)f,\cr
&\eta_t+v\cdot\nabla\eta+a\divv v+\eta\divv v=0,\cr
&v|_{t=0}=v_0,\ \ \eta|_{t=0}=\eta_0,\cr}
\label{1.4}
\end{equation}
where $p_\varrho={dp\over d\varrho}$, $a_0=p_\varrho(a)$.

\noindent
To show stability of incompressible motions in the set of weakly compressible barotropic motions we introduce the quantities
\begin{equation}
u=v-V,\quad q=p-P,\quad g=f-F
\label{1.5}
\end{equation}
Then the quantities $u$, $q$, $g$ satisfy
\begin{equation}\eqal{
&au_t+\eta(v_t+v\cdot\nabla v)+a(v\cdot\nabla u+u\cdot\nabla V)-\mu\Delta u-\nu\nabla\divv v+\nabla q\cr
&=ag+\eta f,\cr
&u|_{t=0}=v_0-V_0\equiv u_0,\cr}
\label{1.6}
\end{equation}
where we used that $v\cdot\nabla v-V\cdot\nabla V=v\cdot\nabla u+u\cdot\nabla V=v\cdot\nabla u+u\cdot\nabla(v-u)$.

\noindent
Moreover, $\eta$ is a solution to the problem
\begin{equation}\eqal{
&\eta_t+v\cdot\nabla\eta+a\divv v+\eta\divv v=0\cr
&\eta|_{t=0}=\eta_0.\cr}
\label{1.7}
\end{equation}
We introduce potentials $\varphi$ and $\psi$ such that
\begin{equation}
v=\rot\psi+\nabla\varphi+G,
\label{1.8}
\end{equation}
where
$$
G={1\over a|\Omega|}\bigg[-\intop_\Omega\eta vdx+\intop_{\Omega^t}(a+\eta)fdxdt'+\intop_\Omega(a+\eta_0)v_0dx\bigg].
$$
Then problems (\ref{1.6}) and (\ref{1.7}) take the form
\begin{equation}\eqal{
&au_t+\eta(\rot\psi_t+\nabla\varphi_t+G_t+(\rot\psi+\nabla\varphi+G)\cdot\nabla (\rot\psi+\nabla\varphi))\cr
&\quad+a((\rot\psi+\nabla\varphi+G)\cdot\nabla u+u\cdot\nabla V)-\mu\Delta u-\nu\nabla\Delta\varphi+\nabla q\cr
&=ag+\eta f,\cr
&u|_{t=0}=\nabla\varphi_0+\rot\psi_0-V_0,\cr}
\label{1.9}
\end{equation}
and
\begin{equation}\eqal{
&\eta_t+v\cdot\nabla\eta+a\Delta\varphi+\eta\Delta\varphi=0\cr
&\eta|_{t=0}=\eta_0.\cr}
\label{1.10}
\end{equation}
The aim of this paper is deriving such estimate for solutions to the Navier-Stokes equations (\ref{1.1}) that regularity of weak solutions can be proved. We are not able to do it for solutions to (\ref{1.1}) directly.

\noindent
However we proved in \cite{Z1} the existence of global regular solutions to weakly compressible Navier-Stokes system (\ref{1.4}). Having the result from \cite{Z1} we construct system (\ref{1.6}) for differences (\ref{1.5}) with coefficients dependent on regular global solutions to (\ref{1.4}). Therefore solutions to (\ref{1.1}) are approximated by solutions to (\ref{1.4}). Hence we have the system with small data so global estimates for regular solutions are easily derived. We restrict our considerations to derive the estimate for $\|u\|_{L_\infty(\R_+;H^1(\Omega))}$. Having the same estimate for $v$ we obtain in (\ref{3.23}) that $\|V\|_{L_\infty(\R_+;H^1(\Omega))}$ is bounded by data.

To formulate the main result we first recall the theorem on existence of global regular solutions to problem (\ref{1.4}) from \cite{Z1}.

\begin{thA}\label{tA}
Let $\nu>0$, $T>0$ be given. Let $f_g$ be the gradient part of $f$. Let $v=\nabla\varphi+\rot\psi$, $\varrho=a+\eta$, $a$-positive constant be a solution\break to (\ref{1.4}). Let $\eta(0)$, $\nabla\varphi(0)$, $\rot\psi(0)\in\Gamma_1^2(\Omega)$, $|\eta|_\infty<a/2$, $\|\nabla\varphi(0)\|_{\Gamma_1^2(\Omega)}\le c/\sqrt{\nu}$, $\|\rot\psi(0)\|_{\Gamma_1^2(\Omega)}\le c$, $\|\eta(0)\|_{\Gamma_1^2(\Omega)}\le c/\nu$, $f\in L_2(0,T;\Gamma_1^1(\Omega))$,\break $|f_g|_{L_2(0,T;L_{6/5}(\Omega))}\le c/\nu$, $f\in L_6(0,T;L_3(\Omega))\cap L_1(0,T;L_\infty(\Omega))$. Assume that there exist positive constants $\varphi_*$ and $c_1$ such that $c_1\nu^\varkappa\le\varphi_*\le\varphi(0)$, where $\varkappa\in(1/2,1)$. Then for $\nu$ sufficiently large and $T<\nu$ there exists a regular solution to problem (\ref{1.4}) such that
$$\eqal{
&\sqrt{\nu}\nabla\varphi,\rot\psi\in L_\infty(0,T;\Gamma_1^2(\Omega))\cap L_2(0,T;\Gamma_1^3(\Omega)),\cr
&\nu\nabla\varphi\in L_2(0,T;\Gamma_1^3(\Omega)),\quad \nu\eta\in L_\infty(0,T;\Gamma_1^2(\Omega)).\cr}
$$
Hence $v\in{\frak N}(\Omega^t)$, $t\le T$ and the estimate holds
\begin{equation}\eqal{
&\|v\|_{{\frak N}(\Omega^t)}\le\phi(\|\nu\eta(0),\sqrt{\nu}\nabla\varphi(0), \rot\psi(0)\|_{\Gamma_1^2(\Omega)},\nu|f_g|_{L_2(0,T;L_{6/5}(\Omega))},\cr
&\|f\|_{L_2(0,T;\Gamma_1^1(\Omega))},\|f\|_{L_6(0,T;L_3(\Omega))\cap L_1(0,T;L_\infty(\Omega))})\equiv D(0),\quad t\le T,\cr}
\label{1.11}
\end{equation}
where $\phi$ is an increasing positive function.\\
Assuming the decay $\|f(t)\|_1\le f_0e^{-\alpha t}$, $f_0=\const$, $\alpha>0$, and that $D(kT)$ is finite, where interval $(0,T)$ is replaced lby $(kT,(k+1)T)$, $k\in\N_0$ and assuming that $T$ is such that
$$
-{a_*\over 2}T+c\intop_0^T(|v(t)|_{3,1}^2+|\Delta\varphi(t)|_\infty^2+ \|v(t)\|_1^4+\|f(t)\|_1^2)dt\le 0,
$$
where $a_*=\min\{a_0,\mu/a\}$, $a_0=p_\varrho(a)$, we obtain that
$$
\|v\|_{{\frak N}(\Omega\times(kT,(k+1)T))}\le D(kT).
$$
\end{thA}

\begin{assA}\label{aA}
Let $(v,\eta)$ be a solution to problem (\ref{1.4}) described by Theorem A. Let $v$ be described by potential $\varphi$ and $\psi$ by (\ref{1.8}). Let $T>0$ be given and
$$
B_2^2(t)=\|v(t)\|_2^2+\|\nabla\varphi(t)\|_2^2
$$
and
$$
\intop_{kT}^{(k+1)T}B_2^2(t)dt\le c\bigg(D^2(kT)+ {D^2(kT)\over\nu^2}+A_1^2A_2^2\bigg), \quad k\in\N_0,
$$
where $D(kT)$ is defined in Theorem A and $A_1$ is defined in Lemma 2.1 in \cite{Z1},
$$
A_2=|f|_{18/7,6,\Omega\times(kT,(k+1)T)}+|\varrho_0|_\infty^{1/6}|v_0|_6.
$$
\end{assA}

\begin{thB}\label{tB}
Let Assumption \ref{aA} hold. Let $(u,\eta)$ satisfy problem (\ref{1.9}), (\ref{1.10}). Assume that
$$
|u_x(0)|_2^2\le\gamma\in(0,\gamma_*],
$$
where $\gamma_*$ is so small that
$$
c\exp\bigg(2c\intop_{kT}^{(k+1)T}B_2^2(t)dt\bigg)\gamma_*^2\le\mu/2,
$$
where $T$ is the time of local solutions for any finite interval $[kT,(k+1)T]$. Assume that
$$
\|g_r(t)\|_1^2\le\gamma_0^2\exp(-\alpha t),
$$
where $g$ is defined in (\ref{1.5}) and $g_r$ is the rotational part of $g$, $\gamma_0$, $\alpha$ are some constants.\\
Then for sufficiently small $\gamma$, $\gamma_0$ and sufficiently large $T$ we have
$$
\|u(t)\|_1^2\le\gamma\exp\bigg[2c\intop_{kT}^{(k+1)T}B_2^2(t)dt\bigg],\quad k\in\N_0,\ \ t\in[kT,(k+1)T].
$$
Then for solutions to problem (\ref{1.1}) we have
$$
\|V(t)\|_1^2\le\gamma\exp\bigg[2c\intop_{kT}^tB_2^2(t')dt'\bigg]+D^2(kT),
$$
for $t\in[kT,(k+1)T]$, where $T$ is defined in Theorem A.
\end{thB}

\begin{reC}%Remark C
We hope that the paper meets one of the statements from \cite{F}.
\end{reC}

\noindent
There is a huge literature concerning the regularity problem of weak solutions to the Navier-Stokes equations. Therefore we are not able to present all papers devoted to this problem. Moreover, we are not able to close the list of mathematicians trying to solve it. Hence, we concdentrate the presentation on some directions and recall mathematicians working in these areas.
\vskip6pt

\begin{itemize}
\item[1.] Formulation of sufficient conditions guaranteeing regularity of weak solutions.\\
The first who formulated such conditions was J. Serrin \cite{S}. This approach was continued by D. Chae, H.J. Choe, H. Kozono, H. Sohr, J. Neustupa, P. Penel and the references of their papers are cited in \cite{Z3, Z4}. We have to recall results of G. Seregin, V. \v Sver\'ak, T. Shilkin, A. Mikhaylov (see \cite{S1, S2, S3, S4, SSS, SS1, ESS, MS}).
\item[2.] Local regularity theory.\\
The direction of examining regularity of weak solutions of the Navier-Stokes equations was initiated in the celebrated paper of L. Caffarelli, R. Kohn, L. Nirenberg (see \cite{CKN}). The famous mathematicians working in this directions are G. Seregin \cite{S1, S2, S3, S4}, V. \v Sver\'ak \cite{SS1}.
\item[3.] Rotating Navier-Stokes equations.\\
The existence of global regular solutions to the rotating Navier-Stokes equations was strongly examined by A. Babin, A. Mahalov, B. Nicolaenko (see \cite{BMN1, BMN2, BMN3, MN}).
\item[4.] Global regular solutions to the Navier-Stokes equations with some special properties. We have to distinguish the following directions
\begin{itemize}
\item[a.] Thin domains (see \cite{RS1, RS2, RS3}).
\item[b.] Small variations of vorticity (see \cite{CF}).
\item[c.] Motions in cylindrical domains with data close to data of 2d solutions (see \cite{Z5, Z6, Z7, Z8, NZ}).
\item[d.] Motions in axisymmetric domains with data close to data of axisymmetric solutions (see \cite{Z3, Z4}).
\end{itemize}
\end{itemize}

\section{Notation and auxiliary results}\label{s2}

We introduce the following simplified notation:
$$
\|u\|_{L_p(\Omega)}=|u|_p,\quad p\in[1,\infty],\quad \|u\|_{H^s(\Omega)}=\|u\|_s,\quad s\in\R_+.
$$
Moreover, we introduce
$$\eqal{
&|u|_{k,l}^2=\sum_{i=0}^l\|\partial_t^iu\|_{k-i},\quad &|u|_{k,l,r,\Omega^t}=\bigg(\intop_0^t|u(t')|_{k,l}^rdt'\bigg)^{1/r},\cr
&|u|_{r,q,\Omega^t}=\bigg(\intop_0^t|u(t')|_r^q\bigg)^{1/q},\quad &\|u\|_{L_r(0,t;H^s(\Omega))}=\|u\|_{s,r,\Omega^t}.\cr}
$$
Let $v$ be defined in the form (\ref{1.8}). We say that $v\in{\frak M}(\Omega^T)$ if
$$\eqal{
&\|v\|_{{\frak M}(\Omega^T)}=\esssup_{t\le T}\bigg(|\nabla\varphi(t)|_{2,1}^2+ {1\over\nu}|\rot\psi(t)|_{2,1}^2\bigg)^{1/2}\cr
&\quad+\bigg(\intop_0^T\bigg(\nu|\nabla\varphi(t)|_{3,1}^2+{1\over\nu} |\rot\psi(t)|_{3,1}^2\bigg)dt\bigg)^{1/2}<\infty.\cr}
$$
Next we say that $v\in{\frak N}(\Omega^T)$ if
$$\eqal{
\|v\|_{{\frak N}(\Omega^T)}&=\esssup_{t\le T} (\nu|\nabla\varphi(t)|_{2,1}^2+|\rot\psi(t)|_{2,1}^2)^{1/2}\cr
&\quad+\bigg(\intop_0^T (\nu^2|\nabla\varphi(t)|_{3,1}^2+|\rot\psi(t)|_{3,1}^2)dt\bigg)^{1/2}<\infty.\cr}
$$
If $\|v\|_{{\frak M}(\Omega^T)}\le D_0$ then $\|v\|_{{\frak N}(\Omega^T)}\le(1+\sqrt{\nu})D_0$.

\noindent
Let
$$
\Gamma_l^k(\Omega)=\{u:|u|_{k,l}<\infty\},\quad l\le k,\ \ l,k\in\N_0.
$$
To apply the Poincar\'e inequality we need

\begin{remark}\label{r2.1}(see [Z1, Lemma 2.1])
Let $(\varrho,v)$ be a solution to problem (\ref{1.2}). Assume that $p=p(\varrho)=A\varrho^\varkappa$, $\varkappa>1$, $\varrho_0\in L_1(\Omega)$, $f\in L_{\infty,1}(\Omega^t)$, $\intop_\Omega\big({1\over2}\varrho_0v_0^2+ {A\over\varkappa-1}\varrho_0^\varkappa\big)dx<\infty$. Then
\begin{equation}\eqal{
&\intop_\Omega\bigg({1\over2}\varrho v^2+{A\over\varkappa-1}\varrho^\varkappa\bigg)dx+\mu|\nabla v|_{2,\Omega^t}^2+ \nu|\divv v|_{2,\Omega^t}^2\cr
&\le2|\varrho_0|_1|f|_{\infty,1,\Omega^t}^2+{3\over2}\intop_\Omega\bigg({1\over2} \varrho_0v_0^2+{A\over\varkappa-1}\varrho_0^\varkappa\bigg)dx\equiv
\bar A_1^2.\cr}
\label{2.1}
\end{equation}
\end{remark}

\noindent
From (\ref{2.1}) and for $|\eta|\le a/2$ we obtain
\begin{equation}
|v|_2^2+\nu|\Delta\varphi|_{2,\Omega^t}^2\le c(a)\bar A_1^2.
\label{2.2}
\end{equation}
Our aim is to find an estimate for $\intop_\Omega udx$.

\noindent
Multiply $(\ref{1.2})_2$ by $v$, add to $(\ref{1.2})_1$ and integrate over $\Omega$. Using the periodic boundary conditions we have
\begin{equation}
{d\over dt}\intop_\Omega\varrho vdx=\intop_\Omega\varrho fdx.
\label{2.3}
\end{equation}
Consider problem (\ref{1.1}). Using the periodic boundary conditions we obtain
\begin{equation}
{d\over dt}\intop_\Omega Vdx=\intop_\Omega Fdx.
\label{2.4}
\end{equation}
Equations (\ref{2.3}) and (\ref{2.4}) imply
\begin{equation}
{d\over dt}\intop_\Omega(\varrho v-aV)dx=\intop_\Omega(\varrho f-aF)dx.
\label{2.5}
\end{equation}
Hence, it follows
\begin{equation}
{d\over dt}\intop_\Omega(au+\eta v)dx=\intop_\Omega(ag+\eta f)dx.
\label{2.6}
\end{equation}
Integrating (\ref{2.6}) with respect to time yields
\begin{equation}\eqal{
\intop_\Omega udx&=-{1\over a}\intop_\Omega\eta vdx+{1\over a}\intop_{\Omega^t}(ag+\eta f)dxdt'\cr
&\quad+\intop_\Omega u_0dx+{1\over a}\intop_\Omega\eta_0v_0dx.\cr}
\label{2.7}
\end{equation}
Hence
\begin{equation}\eqal{
\bigg|\intop_\Omega udx\bigg|&\le{1\over a}(|\eta|_2|v|_2+|\eta|_{2,\infty,\Omega^t}|f|_{2,1,\Omega^t}+|\eta_0|_2|v_0|_2)\cr
&\quad+\bigg|\intop_{\Omega^t}gdxdt'\bigg|+\bigg|\intop_\Omega u_0dx\bigg|\le {c\over a}(|\eta|_{2,\infty,\Omega^t}+|\eta_0|_2)\bar A_1\cr
&\quad+\bigg|\intop_{\Omega^t}gdxdt'\bigg|+\bigg|\intop_\Omega u_0dx\bigg|\cr}
\label{2.8}
\end{equation}
From Lemma 2.1 \cite{Z1},
\begin{equation}
|v|_2^2+\mu\intop_{kT}^t(\|v\|_1^2+\nu|\divv v|_2^2)dt'\le cA_1^2,
\label{2.9}
\end{equation}
where $t\in[kT,(k+1)T]$, $k\in\N_0$.
\goodbreak

\section{Estimates}\label{s3}

First we obtain the energy estimate for solutions to problems (\ref{1.9}), (\ref{1.10}).

\begin{lemma}\label{l3.1}
Assume that $\nabla\varphi\in H^1(\Omega)$, $v\in L_6(\Omega)$, $\eta\in H^1(\Omega)$, $v_t\in L_2(\Omega)$, $g_r\in L_2(\Omega)$, $f\in L_3(\Omega)$ and $\divv g_r=0$. Assume that $A_1$ is the bound of the energy inequality for solutions to problem (\ref{1.2}), $\eta\in L_\infty(\Omega^t)$, $\Delta\varphi\in L_2(\Omega^t)$, $v_t\in L_2(\Omega^t)$, $g\in L_2(\Omega^t)$ and $\big|\intop_\Omega u_0dx\big|<\infty$. Then
\begin{equation}\eqal{
&{d\over dt}|u|_2^2+\mu\|u\|_1^2\le c|u|_3^2(\|\nabla\varphi\|_1^2+|v|_6^2)+c(|\nabla\varphi_t|_2^2+ \|\nabla\varphi\|_1^2\cr
&\quad+\|\nabla\varphi\|_1^4)+c[|v|_3^2\|\nabla\varphi\|_1^2+ |v|_6^4\|\eta\|_1^2\cr
&\quad+|v|_6^2(\|\eta\|_1^2\|\nabla\varphi\|_1^2+\|\eta\|_1^2+ \|\nabla\varphi\|_1^2)+c|v_t|_2^2|\eta|_3^2\cr
&\quad+c|g_r|_2^2+ c|\eta|_6^2|f|_3^2]+c\bigg[(|\eta|_2^2+|\eta_0|_2^2)A_1^2+ \bigg|\intop_{\Omega^t}gdxdt'\bigg|^2+\bigg|\intop_\Omega u_0dx\bigg|^2\bigg].\cr}
\label{3.1}
\end{equation}
\end{lemma}

\begin{proof}
Let $\bar u=u-\nabla\varphi$. Then $\divv\bar u=0$. Multiply (\ref{1.6}) by $\bar u$ and integrate over $\Omega$. Then we have
\begin{equation}\eqal{
&a\intop_\Omega u_t\cdot\bar udx+\intop_\Omega\eta(v_t+v\cdot\nabla v)\cdot\bar udx+a\intop_\Omega v\cdot\nabla u\cdot\bar udx\cr
&\quad+a\intop_\Omega u\cdot\nabla V\cdot\bar udx-\mu\intop_\Omega\Delta u\cdot\bar udx=a\intop_\Omega g\cdot\bar udx+\intop_\Omega\eta f\cdot\bar udx.\cr}
\label{3.2}
\end{equation}
Now, we examine the particular terms in (\ref{3.2}). The first term on the l.h.s. of (\ref{3.2}) equals
$$
{1\over2}{d\over dt}\intop_\Omega|u|^2dx-\intop_\Omega u_t\cdot\nabla\varphi dx,
$$
where integration by parts in the second term implies
$$
-\intop_\Omega u_t\cdot\nabla\varphi dx=\intop_\Omega\Delta\varphi_t\varphi dx=-\intop_\Omega\nabla\varphi_t\cdot\nabla\varphi dx.
$$
The third term on the l.h.s. of (\ref{3.2}) takes the form
$$
\intop_\Omega v\cdot\nabla u\cdot udx-\intop_\Omega v\cdot\nabla u\cdot\nabla\varphi dx\equiv I_1+I_2,
$$
where
$$
I_1={1\over2}\intop_\Omega v\cdot\nabla u^2dx=-{1\over2}\intop_\Omega\divv vu^2dx=-{1\over2}\intop_\Omega\Delta\varphi u^2dx.
$$
Hence
$$
|I_1|\le\varepsilon|u|_6^2+c/\varepsilon|u|_3^2|\Delta\varphi|_2^2.
$$
Next,
$$
I_2=-\intop_\Omega v\cdot\nabla(u\cdot\nabla\varphi)dx+\intop_\Omega v\cdot\nabla\nabla\varphi\cdot udx\equiv I_{21}+I_{22},
$$
where
$$
I_{21}=\intop_\Omega\Delta\varphi u\cdot\nabla\varphi dx
$$
and
$$
|I_{21}|\le\varepsilon|u|_6^2+c/\varepsilon|\nabla\varphi|_3^2|\Delta\varphi|_2^2.
$$
Finally,
$$
|I_{22}|\le\varepsilon|u|_6^2+c/\varepsilon|\nabla^2\varphi|_2^2|v|_3^2.
$$
Consider the fourth term on the l.h.s. of (\ref{3.2}). It takes the form
$$\eqal{
&\intop_\Omega u\cdot\nabla(v-u)\cdot\bar udx=\intop_\Omega u\cdot\nabla v\cdot\bar udx-\intop_\Omega u\cdot\nabla u\cdot\bar udx\cr
&=\intop_\Omega u\cdot\nabla v(u-\nabla\varphi)dx-\intop_\Omega u\cdot\nabla u\cdot(u-\nabla\varphi)dx\cr
&=\intop_\Omega u\cdot\nabla v\cdot udx-\intop_\Omega u\cdot\nabla v\cdot\nabla\varphi dx-\intop_\Omega u\cdot\nabla u\cdot udx\cr
&\quad+\intop_\Omega u\cdot\nabla u\cdot\nabla\varphi dx\equiv J_1+J_2+J_3+J_4.\cr}
$$
Integration by parts in $J_1$ yields
$$
J_1=\intop_\Omega u\cdot\nabla(v\cdot u)dx-\intop_\Omega u\cdot\nabla uvdx \equiv J_{11}+J_{12},
$$
where
$$
J_{11}=-\intop_\Omega\Delta\varphi v\cdot udx
$$
and
$$
|J_{11}|\le\varepsilon|u|_6^2+c/\varepsilon|\Delta\varphi|_2^2|v|_3^2.
$$
Next,
$$
|J_{12}|\le\varepsilon|\nabla u|_2^2+c/\varepsilon|u|_3^2|v|_6^2.
$$
Integration by parts in $J_2$ implies
$$\eqal{
J_2&=\intop_\Omega u\cdot\nabla(v\cdot\nabla\varphi)dx-\intop_\Omega u\cdot\nabla\nabla\varphi\cdot vdx\cr
&\equiv-\intop_\Omega\Delta\varphi v\cdot\nabla\varphi dx-\intop_\Omega u\cdot\nabla\nabla\varphi\cdot vdx\cr
&\equiv J_{21}+J_{22},\cr}
$$
where
$$
|J_{21}|\le\varepsilon|\Delta\varphi|_2^2+c/\varepsilon|v|_3^2|\nabla\varphi|_6^2
$$
and
$$
|J_{22}|\le\varepsilon|u|_6^2+c/\varepsilon|\nabla^2\varphi|_2^2|v|_3^2.
$$
Next, we consider $J_3$
$$
J_3=-{1\over2}\intop_\Omega u\cdot\nabla u^2dx={1\over2}\intop_\Omega\Delta\varphi u^2dx.
$$
Hence
$$
|J_3|\le\varepsilon|u|_6^2+c/\varepsilon|u|_3^2|\Delta\varphi|_2^2.
$$
Finally,
$$
|J_4|\le\varepsilon|\nabla u|_2^2+c/\varepsilon|u|_3^2|\nabla\varphi|_6^2.
$$
The last term on the l.h.s. of (\ref{3.2}) equals
$$
\mu|\nabla u|_2^2-\mu\intop_\Omega\nabla u\cdot\nabla^2\varphi dx,
$$
where the second integral is bounded by
$$
\varepsilon|\nabla u|_2^2+c/\varepsilon|\nabla^2\varphi|_2^2.
$$
Consider the first term on the r.h.s. of (\ref{3.2}).

\noindent
Introducing potentials $\psi_g$, $\varphi_g$ such that $g_r=\rot\psi_g$, $g_g=\nabla\varphi_g$ we have
$$
g=f-F=f_r+f_g-F=f_r-F+f_g\equiv g_r+g_g.
$$
Then
$$
\intop_\Omega g\cdot\bar udx=\intop_\Omega g_r\cdot\bar udx=\intop_\Omega g_r\cdot udx-\intop_\Omega g_r\cdot\nabla\varphi dx\equiv K_1+K_2,
$$
where
$$\eqal{
&|K_1|\le\varepsilon|u|_6^2+c/\varepsilon|g_r|_{6/5}^2,\cr
&|K_2|\le|g_r|_2^2+|\nabla\varphi|_2^2.\cr}
$$
The second integral on the r.h.s. of (\ref{3.2}) takes the form
$$
\intop_\Omega\eta f\cdot(u-\nabla\varphi)dx=\intop_\Omega\eta f\cdot udx-\intop_\Omega\eta f\cdot\nabla\varphi dx\equiv K_3+K_4,
$$
where
$$\eqal{
&|K_3|\le\varepsilon|u|_6^2+c/\varepsilon|\eta|_2^2|f|_3^2,\cr
&|K_4|\le|f|_3|\eta|_6|\nabla\varphi|_2.\cr}
$$
Finally, we examine the second term on the l.h.s. of (\ref{3.2})
$$\eqal{
&\intop_\Omega\eta(v_t+v\cdot\nabla v)\cdot(u-\nabla\varphi)dx\cr
&=\intop_\Omega\eta (v_t+v\cdot\nabla v)\cdot udx-\intop_\Omega\eta(v_t+v\cdot\nabla v)\cdot\nabla\varphi dx\cr
&=\intop_\Omega\eta v_t\cdot udx+\intop_\Omega\eta v\cdot\nabla v\cdot udx-\intop_\Omega\eta v_t\cdot\nabla\varphi dx-\intop_\Omega\eta v\cdot\nabla v\cdot\nabla\varphi dx\cr
&\equiv L_1+L_2+L_3+L_4.\cr}
$$
Continuing, we have
$$
|L_1|\le\varepsilon|u|_6^2+c/\varepsilon|v_t|_2^2|\eta|_3^2.
$$
To estimate $L_2$ we integrate by parts to get
$$\eqal{
L_2&=\intop_\Omega v\cdot\nabla v\cdot u\eta dx\cr
&=\intop_\Omega v\cdot\nabla(v\cdot u\eta)dx-\intop_\Omega v\cdot\nabla u\cdot v\eta dx-\intop_\Omega v\cdot\nabla\eta v\cdot udx\cr
&\equiv L_{21}+L_{22}+L_{23}.\cr}
$$
Integrating by parts in $L_{21}$ gives
$$
L_{21}=-\intop_\Omega\Delta\varphi v\cdot u\eta dx.
$$
Hence
$$\eqal{
&|L_{21}|\le\varepsilon|u|_6^2+ c/\varepsilon|v|_6^2|\eta|_6^2|\Delta\varphi|_2^2,\cr
&|L_{22}|\le\varepsilon|\nabla u|_2^2+c/\varepsilon|v|_6^4|\eta|_6^2,\cr
&|L_{23}|\le\varepsilon|u|_6^2+c/\varepsilon|v|_6^4|\nabla\eta|_2^2.\cr}
$$
Next
$$
|L_3|\le|v_t|_2^2|\eta|_3^2+\|\nabla\varphi\|_1^2.
$$
Integrating by parts in $L_4$ yields
$$\eqal{
L_4&=-\intop_\Omega v\cdot\nabla v\cdot\nabla\varphi\eta dx=-\intop_\Omega v\cdot\nabla(v\cdot\nabla\varphi\eta)dx\cr
&\quad+\intop_\Omega v\cdot\nabla\nabla\varphi\cdot v\eta dx+\intop_\Omega v\cdot\nabla\eta v\cdot\nabla\varphi dx\cr
&\equiv L_{41}+L_{42}+L_{43},\cr}
$$
where
$$
L_{41}=\intop_\Omega\divv vv\cdot\nabla\varphi\eta dx=\intop_\Omega\Delta\varphi v\cdot\nabla\varphi\eta dx.
$$
Therefore,
$$
|L_{41}|\le|\Delta\varphi|_2^2+|\nabla\varphi|_6^2|\eta|_6^2|v|_6^2.
$$
Similarly,
$$
|L_{42}|\le|v|_6^2(|\nabla^2\varphi|_2^2+|\eta|_6^2),\quad |L_{43}|\le|v|_6^2(|\nabla\eta|_2^2+|\nabla\varphi|_6^2).
$$
Summarizing the above estimates and assuming that $\varepsilon$ is sufficiently small implies the inequality
$$\eqal{
&{d\over dt}|u|_2^2+\mu|\nabla u|_2^2\le\varepsilon|u|_6^2+ c\bigg({1\over\varepsilon}+1\bigg)|u|_3^2(\|\nabla\varphi\|_1^2+|v|_6^2)\cr
&\quad+c\bigg({1\over\varepsilon}+1\bigg)(|\nabla\varphi_t|_2^2+ \|\nabla\varphi\|_1^2+\|\nabla\varphi\|_1^4)\cr
&\quad+c\bigg[|v|_3^2\|\nabla\varphi\|_1^2+|v|_6^4 \|\eta\|_1^2+|v|_6^2(\|\eta\|_1^2\|\nabla\varphi\|_1^2+ \|\eta\|_1^2+\|\nabla\varphi\|_1^2)\cr
&\quad+c|v_t|_2^2|\eta|_3^2+c|g_r|_2^2+c|\eta|_6^2 |f|_3^2\bigg].\cr}
$$
Applying the Poincar\'e inequality and using that $\varepsilon$ is sufficiently small the above inequality implies
\begin{equation}\eqal{
&{d\over dt}|u|_2^2+\mu\|u\|_1^2\le c|u|_3^2(\|\nabla\varphi\|_1^2+|v|_6^2)+ c(|\nabla\varphi_t|_2^2+\|\nabla\varphi\|_1^2\cr
&\quad+\|\nabla\varphi\|_1^4)+c\bigg[|v|_3^2\|\nabla\varphi\|_1^2+ |v|_6^4\|\eta\|_1^2+ |v|_6^2\|\eta\|_1^2\|\nabla\varphi\|_1^2\cr
&\quad+|v|_6^2(\|\nabla\varphi\|_1^2+\|\eta\|_1^2)+(|v_t|_2^2|\eta|_3^2+ |g_r|_2^2+|\eta|_6^2|f|_3^2)+\bigg|\intop_\Omega udx\bigg|^2\bigg].\cr}
\label{3.3}
\end{equation}
Using (\ref{2.8}) in (\ref{3.3}) implies (\ref{3.1}) and concludes the proof.
\end{proof}

\begin{lemma}\label{l3.2}
Assume that $v\in\Gamma_1^2(\Omega)$, $\nabla\varphi\in\Gamma_1^2(\Omega)$, $\eta\in H^2(\Omega)$, $g_r\in H^1(\Omega)$, $f\in H^1(\Omega)$.\\
Then
\begin{equation}\eqal{
&{d\over dt}|u_x|_2^2+\mu\|\nabla u\|_1^2\le c|u_x|_2^6+c\|u\|_1^2(\|v\|_2^2+ \|\nabla\varphi\|_2^2)+|v|_6^4\|\eta\|_2^2\cr
&\quad+c|v|_6^2(\|\nabla\varphi\|_2^2+\|\eta\|_2^2+\|\nabla\varphi\|_2^2 \|\eta\|_2^2+\|\eta\|_2^2\|v\|_2^2+c\|v\|_2^2\|\nabla\varphi\|_2^2+\|\eta\|_2^2)\cr
&\quad+\|v_t\|_1^2\|\eta\|_2^2+c(\|\nabla\varphi_t\|_1^2+\|\nabla\varphi\|_2^2+ \|g_r\|_1^2\cr
&\quad+\|f\|_1^2\|\eta\|_1^2).\cr}
\label{3.4}
\end{equation}
\end{lemma}

\begin{proof}
Differentiate (\ref{1.6}) with respect to $x$, multiply the result by $\bar u_x$ and integrate over $\Omega$. Then we have
\begin{equation}\eqal{
&a\intop_\Omega u_{xt}\cdot\bar u_xdx+\intop_\Omega[\eta(v_t+v\cdot\nabla v)]_{,x}\cdot\bar u_xdx+a\intop_\Omega[v\cdot\nabla u]_{,x}\cdot\bar u_xdx\cr
&\quad+a\intop_\Omega[u\cdot\nabla(v-u)]_{,x}\cdot\bar u_xdx-\mu\intop_\Omega \Delta u_x\cdot\bar u_xdx\cr
&=a\intop_\Omega g_x\cdot\bar u_xdx+\intop_\Omega(\eta f)_{,x}\bar u_xdx.\cr}
\label{3.5}
\end{equation}
Next, we examine the particular terms in (\ref{3.5}). In these considerations we omit $a$. The first term on the l.h.s. of (\ref{3.5}) equals
$$
\intop_\Omega u_{xt}\cdot u_xdx-\intop_\Omega u_{xt}\cdot\nabla\varphi_xdx\equiv I_1+I_2,
$$
where $I_1={1\over2}{d\over dt}|u_x|_2^2$ and we integrate by parts in $I_2$ to derive
$$
I_2=\intop_\Omega\Delta\varphi_{xt}\varphi_xdx=-\intop_\Omega\nabla\varphi_{xt} \cdot\nabla\varphi_xdx.
$$
Hence,
$$
|I_2|\le|\nabla\varphi_{xt}|_2^2+|\nabla\varphi_x|_2^2.
$$
The second term on the l.h.s. of (\ref{3.5}) equals
$$\eqal{
&\intop_\Omega\eta_x(v_t+v\cdot\nabla v)\cdot\bar u_xdx+\intop_\Omega\eta(v_{xt}+v_x\cdot\nabla v+v\cdot\nabla v_x)\cdot\bar u_xdx\cr
&=\intop_\Omega\eta_xv_t\cdot u_xdx-\intop_\Omega\eta_xv_t\cdot\nabla\varphi_xdx\cr
&\quad+\intop_\Omega\eta_xv\cdot\nabla v\cdot u_xdx-\intop_\Omega\eta_xv\cdot\nabla v\cdot\nabla\varphi_xdx\cr
&\quad+\intop_\Omega\eta v_{xt}\cdot u_xdx-\intop_\Omega\eta v_{xt}\cdot\nabla\varphi_xdx\cr
&\quad+\intop_\Omega\eta v_x\cdot\nabla v\cdot u_xdx-\intop_\Omega\eta v_x\cdot\nabla v\cdot\nabla\varphi_xdx\cr
&\quad+\intop_\Omega\eta v\cdot\nabla v_x\cdot u_xdx-\intop_\Omega\eta v\cdot\nabla v_x\cdot\nabla\varphi_xdx\equiv\sum_{i=1}^{10}J_i.\cr}
$$
Now, we estimate the terms $J_i$, $i=1,\dots,10$. First we have
$$\eqal{
&|J_1|\le\varepsilon|u_x|_6^2+c/\varepsilon|\eta_x|_3^2|v_t|_2^2,\cr
&|J_2|\le|\eta_x|_6|v_t|_3|\nabla\varphi_x|_2\le|v_t|_3^2\|\eta\|_1^2+ \|\nabla\varphi\|_1^2.\cr}
$$
To examine $J_3$ we integrate by parts. Then we obtain
$$\eqal{
J_3&=\intop_\Omega v\cdot\nabla v\cdot u_x\eta_xdx=\intop_\Omega v\cdot\nabla(v\cdot u_x\eta_x)dx-\intop_\Omega v\cdot\nabla u_x\cdot v\eta_xdx\cr
&\quad-\intop v\cdot\nabla\eta_xv\cdot u_xdx=J_{31}+J_{32}+J_{33}.\cr}
$$
Integration by parts in $J_{31}$ implies
$$
J_{31}=-\intop_\Omega\Delta\varphi v\cdot u_x\eta_xdx.
$$
Hence
$$
|J_{31}|\le\varepsilon|u_x|_6^2+c/\varepsilon|\eta_x|_6^2|\Delta\varphi|_2^2 |v|_6^2.
$$
Next,
$$\eqal{
&|J_{32}|\le\varepsilon|\nabla u_x|_2^2+c/\varepsilon|v|_6^4|\eta_x|_6^2,\cr
&|J_{33}|\le\varepsilon|u_x|_6^2+c/\varepsilon|v|_6^4|\nabla\eta_x|_2^2.\cr}
$$
To examine $J_4$ we integrate by parts. Then we have
$$\eqal{
J_4&=-\intop_\Omega v\cdot\nabla v\cdot\nabla\varphi_x\eta_xdx=-\intop_\Omega v\cdot\nabla(v\cdot\nabla\varphi_x\eta_x)dx\cr
&\quad+\intop_\Omega v\cdot\nabla\nabla\varphi_x\cdot v\eta_xdx
+\intop_\Omega v\cdot\nabla\eta_x\nabla\varphi_x\cdot vdx\cr
&\equiv J_{41}+J_{42}+J_{43}.\cr}
$$
Integrating by parts in $J_{41}$ implies
$$
J_{41}=\intop_\Omega\Delta\varphi v\cdot\nabla\varphi_x\eta_xdx.
$$
Therefore,
$$
|J_{41}|\le|\Delta\varphi|_2^2+|v|_6^2|\eta_x|_6^2|\nabla\varphi_x|_6^2.
$$
Continuing
$$
|J_{42}|\le|v|_6^2(|\nabla^2\varphi_x|_2|\eta_x|_6)\le|v|_6^2 (|\nabla\varphi_{xx}|_2^2+\|\eta\|_2^2)
$$
and
$$
|J_{43}|\le|v|_6^2|\nabla\eta_x|_2|\nabla\varphi_x|_6\le|v|_6^2(\|\eta\|_2^2+ \|\nabla\varphi\|_2^2).
$$
Estimating $J_5$ yields
$$
|J_5|\le\varepsilon|u_x|_6^2+c/\varepsilon|\eta|_3^2|v_{xt}|_2^2.
$$
Next
$$
|J_6|\le|\eta|_\infty^2|v_{xt}|_2^2+|\nabla\varphi_x|_2^2\le\|v_t\|_1^2 \|\eta\|_2^2+\|\nabla\varphi\|_1^2.
$$
Integration by parts in $J_7$ implies
$$\eqal{
J_7&=\intop_\Omega v_x\cdot\nabla(v\cdot u_x\eta)dx-\intop_\Omega v_x\cdot\nabla u_x\cdot v\eta dx-\intop_\Omega v_x\cdot\nabla\eta u_x\cdot vdx\cr
&\equiv J_{71}+J_{72}+J_{73},\cr}
$$
where
$$
J_{71}=-\intop_\Omega\divv v_xv\cdot u_x\eta dx=-\intop_\Omega\Delta\varphi_xv\cdot u_x\eta dx.
$$
Hence
$$
|J_{71}|\le\varepsilon|u_x|_6^2+c/\varepsilon|\Delta\varphi_x|_2^2|v|_6^2|\eta|_6^2
$$
and
$$\eqal{
&|J_{72}|\le\varepsilon|\nabla u_x|_2^2+c/\varepsilon|v_x|_6^2|\eta|_6^2|v|_6^2,\cr
&|J_{73}|\le\varepsilon|u_x|_6^2+c/\varepsilon|\nabla\eta|_6^2|v_x|_2^2|v|_6^2.\cr}
$$
Continuing, we integrate by parts in $J_8$ to derive
$$\eqal{
J_8&=-\intop_\Omega v_x\cdot\nabla(v\cdot\nabla\varphi_x\eta)dx+\intop_\Omega v_x\cdot\nabla\nabla\varphi_x\cdot v\eta dx+\intop_\Omega v_x\cdot\nabla\eta\nabla\varphi_x\cdot vdx\cr
&\equiv J_{81}+J_{82}+J_{83}.\cr}
$$
Integration by parts in $J_{81}$ yields
$$
J_{81}=\intop_\Omega\Delta\varphi_xv\cdot\nabla\varphi_x\eta dx,
$$
so
$$
|J_{81}|\le|\Delta\varphi_x|_2^2+|v|_6^2|\nabla\varphi_x|_6^2|\eta|_6^2.
$$
Estimations of other terms in $J_8$ implies
$$\eqal{
&|J_{82}|\le|\nabla^2\varphi_x|_2^2+|v_x|_3^2|\eta|_\infty^2|v|_6^2,\cr
&|J_{83}|\le|\nabla\varphi_x|_6^2+|v_x|_2^2|\nabla\eta|_6^2|v|_6^2.\cr}
$$
Finally,
$$\eqal{
&|J_9|\le\varepsilon|u_x|_6^2+c/\varepsilon|\nabla v_x|_2^2|v|_6^2|\eta|_6^2,\cr
&|J_{10}|\le|\nabla\varphi_x|_6^2+c|\nabla v_x|_2^2|v|_6^2|\eta|_6^2.\cr}
$$
The third term on the l.h.s. of (\ref{3.5}) takes the form
$$\eqal{
&\intop_\Omega v_x\cdot\nabla u\cdot u_xdx+\intop_\Omega v\cdot\nabla u_x\cdot u_xdx-\intop_\Omega v_x\cdot\nabla u\cdot\nabla\varphi_xdx\cr
&\quad-\intop_\Omega v\cdot\nabla u_x\cdot\nabla\varphi_xdx\equiv K_1+K_2+K_3+K_4.\cr}
$$
Integrating by parts with respect to $x$ in $K_1$ yields
$$
K_1+K_2=-\intop_\Omega v\cdot\nabla u\cdot u_{xx}dx.
$$
Then
$$
|K_1+K_2|\le\varepsilon|u_{xx}|_2^2+c/\varepsilon|\nabla u|_2^2|v|_\infty^2.
$$
Similarly,
$$
K_3+K_4=\intop_\Omega v\cdot\nabla u\cdot\nabla\varphi_{xx}dx
$$
so
$$
|K_3+K_4|\le\varepsilon|\nabla u|_6^2+c/\varepsilon|\nabla\varphi_{xx}|_2^2|v|_3^2.
$$
The fourth term on the l.h.s. of (\ref{3.5}) has the form
$$\eqal{
&\intop_\Omega(u\cdot\nabla v)_x\cdot(u_x-\nabla\varphi_x)dx-\intop_\Omega(u\cdot\nabla u)_x(u_x-\nabla\varphi_x)dx\cr
&=\intop_\Omega(u\cdot\nabla v)_x\cdot u_xdx-\intop_\Omega(u\cdot\nabla v)_x\cdot\nabla\varphi_xdx-\intop_\Omega(u\cdot\nabla u)_x\cdot u_xdx\cr
&\quad+\intop_\Omega(u\cdot\nabla u)_x\cdot\nabla\varphi_xdx\equiv L_1+L_2+L_3+L_4.\cr}
$$
First we consider
$$
L_1=\intop_\Omega u_x\cdot\nabla v\cdot u_xdx+\intop_\Omega u\cdot\nabla v_x\cdot u_xdx\equiv L_{11}+L_{12},
$$
where
$$\eqal{
&|L_{11}|\le\varepsilon|u_x|_6^2+c/\varepsilon|u_x|_2^2|\nabla v|_3^2,\cr
&|L_{12}|\le\varepsilon|u_x|_6^2+c/\varepsilon|\nabla v_x|_2^2|u|_3^2.\cr}
$$
Next, we examine
$$
L_2=-\intop_\Omega u_x\cdot\nabla v\cdot\nabla\varphi_xdx-\intop_\Omega u\cdot\nabla v_x\cdot\nabla\varphi_xdx\equiv L_{21}+L_{22},
$$
where
$$\eqal{
&|L_{21}|\le\varepsilon|u_x|_6^2+c/\varepsilon|\nabla v|_2^2|\nabla\varphi_x|_3^2,\cr
&|L_{22}|\le\varepsilon|u|_\infty^2+c/\varepsilon|\nabla v_x|_2^2|\nabla\varphi_x|_2^2.\cr}
$$
Next, we examine $L_3$. We express it in the form
$$
L_3=-\intop_\Omega u\cdot\nabla u_x\cdot u_xdx-\intop_\Omega u_x\cdot\nabla u\cdot u_xdx\equiv L_{31}+L_{32},
$$
where
$$
L_{31}=-{1\over2}\intop_\Omega u\cdot\nabla|u_x|^2dx={1\over2}\intop_\Omega\Delta\varphi u_x^2dx.
$$
Hence
$$
|L_{31}|\le\varepsilon|u_x|_6^2+c/\varepsilon|\Delta\varphi|_3^2|u_x|_2^2
$$
and
$$
|L_{32}|\le|u_x|_3^3\le c|u_{xx}|_2^{3/2}|u_x|_2^{3/2}\le\varepsilon|u_{xx}|_2^2+c/\varepsilon|u_x|_2^6.
$$
Finally, we examine
$$
L_4=\intop_\Omega u\cdot\nabla u_x\cdot\nabla\varphi_xdx+\intop_\Omega u_x\cdot\nabla u\cdot\nabla\varphi_xdx\equiv L_{41}+L_{42}.
$$
Continuing, we have
$$\eqal{
&|L_{41}|\le\varepsilon|\nabla u_x|_2^2+c/\varepsilon|\nabla\varphi_x|_3^2|u|_6^2,\cr
&|L_{42}|\le\varepsilon|\nabla u|_6^2+c/\varepsilon|u_x|_2^2|\nabla\varphi_x|_3^2.\cr}
$$
The last term on the l.h.s. of (\ref{3.5}) equals
$$
\mu\intop_\Omega|\nabla u_x|^2dx+\mu\intop_\Omega\Delta u_x\nabla\varphi_xdx,
$$
where the second term is treated in the way
$$\eqal{
\mu\intop_\Omega\Delta u_x\cdot\nabla\varphi_xdx&=-\mu\intop_\Omega\Delta\divv u_x\varphi_xdx=-\mu\intop_\Omega\Delta^2\varphi_x\varphi_xdx\cr
&= -\intop_\Omega|\Delta\varphi_x|^2dx.\cr}
$$
The first term on the r.h.s. of (\ref{3.5}) equals
$$
\intop_\Omega g_{rx}(u-\nabla\varphi)_xdx=\intop_\Omega g_{rx}u_xdx-\intop_\Omega g_{rx}\nabla\varphi_xdx\equiv M_1+M_2,
$$
where
$$\eqal{
&|M_1|\le\varepsilon|u_x|_2^2+c/\varepsilon|g_{rx}|_2^2,\cr
&|M_2|\le|g_{rx}|_2^2+|\nabla\varphi_x|_2^2.\cr}
$$
Finally, the last term on the r.h.s. of (\ref{3.5}) assumes the form
$$
\intop_\Omega(\eta_xf+\eta f_x)\cdot u_xdx-\intop_\Omega(\eta_xf+\eta f_x)\cdot\nabla\varphi_xdx\equiv N_1+N_2,
$$
where
$$\eqal{
&|N_1|\le\varepsilon|u_x|_6^2+c/\varepsilon(|\eta_x|_2^2|f|_3^2+ |\eta|_3^2|f_x|_2^2),\cr
&|N_2|\le|\nabla\varphi_x|_6\|\eta\|_1\|f\|_1.\cr}
$$
Employing the above estimates in (\ref{3.5}) and using that $\varepsilon$ is sufficiently small we obtain the inequality
\begin{equation}\eqal{
&{d\over dt}|u_x|_2^2+\mu\|u_x\|_1^2\le c|u_x|_2^6+c\|u\|_1^2(\|v\|_2^2+\|\nabla\varphi\|_2^2)+|v|_6^4 \|\eta\|_2^2\cr
&\quad+c\|v\|_2^2(\|\nabla\varphi\|_2^2+\|\eta\|_2^2)+c|v|_6^2(\|\nabla\varphi\|_2^2+\|\eta\|_2^2 +\|\nabla\varphi\|_2^2\|\eta\|_2^2\cr
&\quad+\|\eta\|_2^2\|v\|_2^2)+c\|v_t\|_1^2\|\eta\|_2^2+c(\|\nabla\varphi_t\|_1^2+ \|\nabla\varphi\|_2^2+\|g_r\|_1^2+\|f\|_1^2\|\eta\|_1^2).\cr}
\label{3.6}
\end{equation}
This implies (\ref{3.4}) and concludes the proof.
\end{proof}

\noindent
Let
$$\eqal{
B_1(t)&=\|\nabla\varphi(t)\|_1^2+|v(t)|_6^2,\cr
G_1^2(t)&=|\nabla\varphi_t|_2^2+\|\nabla\varphi\|_1^2+\|\nabla\varphi\|_1^4+ |v|_6^2(\|\nabla\varphi\|_1^2\cr
&\quad+\|\eta\|_1^2+\|\eta\|_1^2\|\nabla\varphi\|_1^2)+|v|_6^4\|\eta\|_1^2+ |v_t|_2^2|\eta|_3^2+|g_r|_2^2\cr
&\quad+|\eta|_6^2|f|_3^2+c\bigg[(|\eta|_2^2+|\eta_0|_2^2)A_1^2+ \bigg|\intop_{\Omega^t}gdxdt'\bigg|^2\bigg]+\bigg|\intop_\Omega u_0dx\bigg|^2.\cr}
$$

\begin{remark}\label{r3.3}
In view of the above notation we express (\ref{3.1}) in the form
\begin{equation}
{d\over dt}|u|_2^2+\mu\|u\|_1^2\le c|u|_3^2B_1(t)+cG_1^2(t)
\label{3.7}
\end{equation}
In view of interpolation
$$
|u|_3\le c|u_x|_2^{1/2}|u|_2^{1/2}
$$
we derive from (\ref{3.7}) the inequality
\begin{equation}
{d\over dt}|u|_2^2+\mu\|u\|_1^2\le c|u|_2^2B_1^2+cG_1^2(t).
\label{3.8}
\end{equation}
Consider (\ref{3.8}) in the time interval $[kT,(k+1)T]$, $k\in\N_0=\N\cup\{0\}$. Moreover, we assume that in each time interval $[kT,(k+1)T]$ the first part of Theorem B holds which is proved step by step in time.

\noindent
We use that
$$\eqal{
&\intop_{kT}^{(k+1)T}B_1^2(t)dt\le\sup_{t\in[kT,(k+1)T]}\|\nabla\varphi(t)\|_1^2 \intop_{kT}^{(k+1)T}\|\nabla\varphi(t)\|_1^2dt\cr
&\quad+\sup_{t\in[kT,(k+1)T]}|v(t)|_6^2\intop_{kT}^{(k+1)T}|v(t)|_6^2dt\le {c\over\nu^4}+A_1^2\bigg({c\over\nu^{2/3}}+{c\over\nu^{2\alpha/3}}+A_2^2\bigg),\cr}
$$
where we used estimate for $\sup_t|v(t)|_6$ from (2.95) in \cite{Z1}, $\alpha>0$ and $A_2=|f|_{18/7,6,\Omega^t}+|\varrho_0|_\infty^{1/6}|v_0|_6$.

\noindent
Moreover
$$
\intop_{kT}^{(k+1)T}G_1^2(t)dt\le{c\over\nu}+ce^{-\alpha kT},
$$
where we used that $|\eta(0)|_{2,1}\le{c\over\nu}$, $|g(t)|_2\le\gamma_0e^{-\alpha t}$. We express (\ref{3.8}) in the form
$$
{d\over dt}|u|_2^2+\mu_1|u|_2^2+\mu_2\|u\|_1^2\le c|u|_2^2B_1^2+cG_1^2,
$$
where $\mu=\mu_1+\mu_2$, $\mu_i>0$, $i=1,2$. Then we obtain
$$\eqal{
&{d\over dt}\bigg [|u|_2^2\exp\bigg(\mu_1(t-kT)-c\intop_{kT}^tB_1^2(t')dt'\bigg)\bigg]\cr
&\quad+\mu_2\|u\|_1^2\exp\bigg(\mu_1(t-kT)-c\intop_{kT}^tB_1^2(t')dt'\bigg)\cr
&\le cG_1^2\exp\bigg(\mu_1(t-kT)-c\intop_{kT}^tB_1^2(t')dt'\bigg).\cr}
$$
Integrating the above inequality from $t=kT$ to $t\in(kT,(k+1)T]$ we derive
\begin{equation}\eqal{
&|u(t)|_2^2+\mu_2\intop_{kT}^t\|u(t')\|_1^2\exp[\mu_1(t'-t)]dt'\cr
&\le c\exp\bigg(c\intop_{kT}^tB_1^2(t')dt'\bigg)\intop_{kT}^tG_1^2(t')dt'\cr
&\quad+c |u(kT)|_2^2\exp\bigg[-\mu_1(t-kT)+c\intop_{kT}^tB_1^2(t')dt'\bigg].\cr}
\label{3.9}
\end{equation}
Setting $t=(k+1)T$ we obtain
$$\eqal{
|u((k+1)T)|_2^2&\le c\exp\bigg(c\intop_{kT}^{(k+1)T}B_1^2(t)dt\bigg)\intop_{kT}^{(k+1)T}G_1^2(t)dt\cr
&\quad+c\exp\bigg(-\mu_1T+c\intop_{kT}^{(k+1)T}B_1^2(t)dt\bigg)\cdot|u(kT)|_2^2.\cr}
$$
Since
$$
\intop_{kT}^{(k+1)T}B_1^2(t)dt\le{c\over\nu^2}+A_1^2A_2^2,
$$
$T$ close to $\nu$ and $\nu$ large we have that there exists a constant $\mu_0>0$ such that
$$
-\mu_1T+c\bigg({1\over\nu}+A_1^2A_2^2\bigg)\le-\mu_0T.
$$
Therefore, we have
$$\eqal{
&|u((k+1)T)|_2^2\le c\exp\bigg(c\bigg({1\over\nu^2}+A_1^2A_2^2\bigg)\bigg) \bigg({1\over\nu}+\gamma_0\exp(-\alpha kT)\bigg)\cr
&\quad+\exp(-\mu_0T)|u(kT)|_2^2.\cr}
$$
Hence, by iteration we get
\begin{equation}\eqal{
|u(kT)|_2^2&\le {c\exp\left (c\left ({1\over\nu^2}+A_1^2A_2^2\right )\right ) \left ({1\over\nu}+\gamma_0\exp(-\alpha kT)\right )\over 1-e^{-\mu_0T}}\cr
&\quad+\exp(-\mu_0kT)|u(0)|_2^2\equiv D_1^2(k).\cr}
\label{3.10}
\end{equation}
Therefore for small $|u(0)|_2$, large $\nu$ and $T$ close to $\nu$ we get that
$$
|u(kT)|_2\le D_1(0)
$$
which is also small and bounded. From (\ref{3.9}) and (\ref{3.10}) we have
\begin{equation}\eqal{
|u(t)|_2^2&\le c\exp\bigg(c\bigg({1\over\nu^2}+A_1^2A_2^2\bigg)\bigg) \bigg({1\over\nu}+\gamma_0\exp(-\alpha kT)\bigg)\cr
&\quad+\exp(-\mu_0T)D_1^2(0)\equiv D_2^2,\quad t\in[kT,(k+1)T].\cr}
\label{3.11}
\end{equation}
Integrating (\ref{3.8}) with respect to time from $kT$ to $t\in(kT,(k+1)T]$, $k\in\N_0$, we obtain
\begin{equation}
|u(t)|_2^2+\mu\intop_{kT}^t\|u(t')\|_1^2dt'\le cD_2^2\intop_{kT}^{(k+1)T}B_1^2(t)dt+c\intop_{kT}^{(k+1)T}G_1^2(t)dt+D_1^2(0).
\label{3.12}
\end{equation}
This ends the Remark.
\end{remark}

\noindent
Next we obtain similar estimates to (\ref{3.11}) and (\ref{3.12}) for $\|u(t)\|_1$. For this purpose we add (\ref{3.1}) and (\ref{3.4}). Then we obtain
\begin{equation}\eqal{
&{d\over dt}\|u\|_1^2+\mu\|u\|_2^2\le c\|u\|_1^6+c\|u\|_1^2(\|v\|_2^2+ \|\nabla\varphi\|_2^2)\cr
&\quad+c|v|_6^4\|\eta\|_2^2+c|v|_6^2(\|\nabla\varphi\|_2^2+\|\eta\|_2^2+ \|\nabla\varphi\|_2^2\|\eta\|_2^2+\|\eta\|_2^2\|v\|_2^2)\cr
&\quad+c\|v_t\|_1^2\|\eta\|_2^2+c\|v\|_2^2(\|\nabla\varphi\|_2^2+\|\eta\|_2^2)+ c(\|\nabla\varphi_t\|_1^2\cr
&\quad+\|\nabla\varphi\|_2^2+\|g_r\|_1^2+\|\eta\|_1^2\|f\|_1^2)+c\bigg[(|\eta|_2^2+ |\eta_0|_2^2)A_1^2\cr
&\quad+\bigg|\intop_{\Omega^t}gdxdt'\bigg|^2+\bigg|\intop_\Omega u_0dx\bigg|^2\bigg]\equiv c\|u\|_1^6+cB_2^2(t)\|u\|_1^2\cr
&\quad+cG_2^2(t).\cr}
\label{3.13}
\end{equation}

\begin{lemma}\label{l3.4}
Assume that $(v,\varrho)$ is a solution to problem (\ref{1.2}), (\ref{1.3}). Assume that $(v,\varrho)$ is described by Theorem A. Assume that $\nu,T$ are sufficiently large, $\|g(t)\|_1\le\gamma_0\exp(-\alpha t)$, $\alpha>0$. Let constant $A$ be introduced in Lemma 4.1 from \cite{Z1}. Assume that $\|u(0)\|_1\le\gamma$, where $\gamma$ is sufficiently small.\\
Then
\begin{equation}
\|u(t)\|_1\le\gamma\exp A.
\label{3.14}
\end{equation}
\end{lemma}

\begin{proof}
To obtain estimate (\ref{3.14}) for $\|u(t)\|_1$ we consider (\ref{3.13}) in the form
\begin{equation}
{d\over dt}\|u\|_1^2\le-(\mu-c\|u\|_1^4)\|u\|_1^2+cB_2^2\|u\|_1^2+cG_2^2.
\label{3.15}
\end{equation}
Consider (\ref{3.15}) in the time interval $[kT,(k+1)T]$. Introduce the quantities
\begin{equation}\eqal{
&X^2(t)=\exp\bigg(-c\intop_{kT}^tB_2^2(t')dt'\bigg)\|u(t)\|_1^2,\cr
&K^2(t)=c\exp\bigg(-c\intop_{kT}^tB_2^2(t')dt'\bigg)G_2^2(t).\cr}
\label{3.16}
\end{equation}
Then (\ref{3.15}) takes the form
\begin{equation}
{d\over dt}X^2\le-\bigg[\mu-c\exp\bigg(2c\intop_{kT}^tB_2^2(t')dt'\bigg) X^4\bigg]X^2+K^2.
\label{3.17}
\end{equation}
From Lemma 4.1 form\cite{Z1} and the considerations of Section 5 from \cite{Z1} we have
\begin{equation}
\intop_{kT}^{(k+1)T}B_2^2(t)dt\le cA^2,
\label{3.18}
\end{equation}
where $T$ is proportional to some increasing function of $\nu$.

\noindent
In view of Lemma 4.1 \cite{Z1} we also have
\begin{equation}
K^2(t)\le c\bigg({A\over\nu^2}+\gamma_0^2e^{-2\alpha t}\bigg).
\label{3.19}
\end{equation}
Now we want to estimate $X(t)$ for $t\in[kT,(k+1)T]$. Assume that
\begin{equation}
X^2(kT)=\|u(kT)\|_1^2\le\gamma.
\label{3.20}
\end{equation}
Suppose that
$$
t_*=\inf\{t\in[kT,(k+1)T]:X^2(t)>\gamma\}
$$
Let $\gamma\in(0,\gamma_*]$, where $\gamma_*$ is so small that
\begin{equation}
\mu-c\exp\bigg(2c\intop_{kT}^{(k+1)T}B_2^2(t')dt'\bigg)\gamma_*^2\ge{\mu\over2}
\label{3.21}
\end{equation}
In view of (\ref{3.18}) condition (\ref{3.21}) takes the form
\begin{equation}
\mu-c\exp(cA^2)\gamma_*^2\ge{\mu\over 2}
\label{3.22}
\end{equation}
Hence, for $t\le t_*$ we derive from (\ref{3.17}) the inequality
\begin{equation}
{d\over dt}X^2\le-{\mu\over 2}X^2+K^2.
\label{3.23}
\end{equation}
Assume that $1\over\nu$, $\gamma_0$ are so small that
$$
K^2(t)\le c\bigg({A\over\nu^2}+\gamma_0^2e^{-2\alpha t}\bigg)\le{\mu\over 4} \gamma\quad {\rm for}\ \ t\in[kT,(k+1)T].
$$
Then
$$
{d\over dt}X^2\big|_{t=t_*}\le-\bigg({\mu\over2}-{\mu\over4}\bigg)\gamma<0,
$$
so $t_*$ does not exist in $[kT,(k+1)T]$. Hence (\ref{3.14}) holds for $t\in[kT,(k+1)T]$ under assumption that (\ref{3.20}) holds. Now we have to show (\ref{3.20}) for any $k\in\N$.

\noindent
Since we showed that
$$
\|u(t)\|_1^2\le\gamma\exp\bigg[c\intop_{kT}^tB_2^2(t')dt'\bigg],\quad t\in[kT,(k+1)T],
$$
we obtain from (\ref{3.15}) the inequality
$$\eqal{
&\|u((k+1)T)\|_1^2\le\exp\bigg[-\mu(k+1)T+c\intop_{kT}^{(k+1)T}B_2^2(t')dt'\cr
&\quad+cT\gamma^2\exp\bigg[2c\intop_{kT}^{(k+1)T}B_2^2(t')dt'\bigg]\cdot \intop_{kT}^{(k+1)T}G_2^2(t')\cdot\cr
&\quad\cdot\exp\bigg[\mu t'-c\intop_{kT}^{t'}B_2^2(t'')dt''-cT\gamma^2\exp \bigg[2c\intop_{kT}^{t'}B_2^2(t'')dt''\bigg]dt'\cr
&\quad+\exp\bigg(-\mu T+c\intop_{kT}^{(k+1)T}B_2^2(t')dt'+cT\gamma^2\exp \bigg[2c\intop_{kT}^{(k+1)T}B_2^2(t')dt'\bigg]\bigg)\cdot\cr
&\quad\cdot\|u(kT)\|_1^2.\cr}
$$
Simplifying, we get
\begin{equation}\eqal{
\|u((k+1)T)\|_1^2&\le\exp[cA^2+cT\gamma^2\exp A^2]\intop_{kT}^{(k+1)T}G_2^2(t)dt\cr
&\quad+\exp[-\mu T+cA^2+cT\gamma^2\exp A^2]\|u(kT)\|_1^2.\cr}
\label{3.24}
\end{equation}
Using that
$$
\intop_{kT}^{(k+1)T}G_2^2(t)dt\le{c\over\nu^2}+\gamma_0\exp(-\alpha kT)
$$
and assuming that $\gamma$ is so small that $T\gamma^2\le c$ we obtain from (\ref{3.24}) the inequality
\begin{equation}\eqal{
\|u((k+1)T)\|_1^2&\le\exp[\phi(A)]\bigg({c\over\nu^2}+\gamma_0\exp(-\alpha kT)\bigg)\cr
&\quad+\exp(-\mu T+\phi(A))\|u(kT)\|_1^2,\cr}
\label{3.25}
\end{equation}
where $\phi(A)=c(A+\exp(cA^2))$.

\noindent
Since $\|u(kT)\|_1^2\le\gamma$ then for a given $A$, sufficiently large $T$, $\nu$ and sufficiently small $\gamma_0$ we have that
$$
\|u((k+1)T)\|_1^2\le\gamma.
$$
Repeating the considerations for any $k\in\N$ we prove the lemma.
\end{proof}

\begin{proof}[Proof of Theorem B]
\begin{equation}
\|V(t)\|_1^2\le\|u\|_1^2+\|v\|_1^2\le\gamma\exp\bigg(c\intop_{kT}^tB_2^2dt'\bigg) +\|v\|_{{\frak N}(\Omega\times(kT,t))},
\label{3.26}
\end{equation}
$t\in[kT,(k+1)T]$, $k\in\N_0$ and $T$ is the time of local solution introduced in Theorem A.
\end{proof}

\bibliographystyle{amsplain}

\begin{thebibliography}{XXX9}
\bibitem[BMN1]{BMN1} Babin, A.; Mahalov, A.; Nicolaenko, B.: Global regularity of 3D rotating Navier-Stokes equations for resonant domains, Appl. Math. Letters 13 (2000), 51--57.
\bibitem[BMN2]{BMN2} Babin, A.; Mahalov, A.; Nicolaenko, B.: Regularity and integrability of 3D Euler and Navier-Stokes equations for rotating fluids, Asympt. Analysis 15 (1997), 103--150.
\bibitem[BMN3]{BMN3} Babin, A.; Mahalov, A.; Nicolaenko, B.: Global regularity of 3D rotating Navier-Stokes equations for resonant domains, Indiana Univ. Math. J. 48 (1999), 1133--1176.
\bibitem[CKN]{CKN} Caffarelli, L.; Kohn, R.; Nirenberg, L.: Partial regularity of suitable weak solutions of the Navier-Stokes equations, Comm. Pure Appl. Math. 35 (1982), 771--831.
\bibitem[CF]{CF} Constantin, P.; Fefferman, C.: Directions of vorticity and the problem of global regularity for the Navier-Stokes equations, Indiana Univ. Math. J. 42 (3) (1993), 775--789.
\bibitem[ESS]{ESS} Escauriaza, L.; Seregin, G.A.; \v Sver\'ak, V.: $L_{3,\infty}$-solutions of the Navier-Stokes equations, Russian Math. Surveys, 58:2 (2003), 211--250.
\bibitem[F]{F} Fefferman, C.L.: Existence and Smoothness of the Navier-Stokes Equation, The Millennium Prize Problems, Clay Mathematics Institute, Cambridge, 57--67.
\bibitem[MS]{MS} Mikhailov, A.S.; Shilkin, T.N.: $L_{3,\infty}$-solutions to the 3d Navier-Stokes system in the domain with a curved boundary, Zap. Nauchn. Sem. POMI 336 (2006), 133--152.
\bibitem[MN]{MN} Mahalov, A.; Nicolaenko, B.: Global solvability of 3D Navier-Stokes equations with a strong initial rotation, Usp. Mat. Nauk 58, 2 (350) (2003), 79--110 (in Russian).
\bibitem[NZ]{NZ} Nowakowski, B.; Zaj\c{a}czkowski, W.M.: Global existence of solutions to Navier-Stokes equations in cylindrical domains, Appl. Math. 36 (2) (2009), 169--182.
\bibitem[RS1]{RS1} Raugel, G.; Sell, G.R.: Navier-Stokes equations on thin 3d domains. I: Global attractors and global regularity of solutions, J. AMS 6 (3) (1993), 503--568.
\bibitem[RS2]{RS2} Raugel, G.; Sell, G.R.: Navier-Stokes equations on thin 3d domains, II: Global regularity on spatially periodic solutions, Nonlinear Partial Differential Equations and Their Applications, Coll\'ege de France Seminar; Vol. 11, eds H. Brezis and J.L. Lions, Pitman Research Notes in Mathematics Series 299, Longman Scientific Technical, Essex UK (1994), 205--247.
\bibitem[RS3]{RS3} Raugel, G.; Sell, G.R.: Navier-Stokes equations on thin 3d domains. III: global and local attractors, Turbulence in Fluid  Flows: A dynamic System Approach, IMA Volumes in Mathematics and its Applications, Vol. 55, eds. G.R. Sell, C. Foias, R. Temam, Springer Verlag, New York 1993, 137--163.
\bibitem[S1]{S1} Seregin, G.A.: Estimates of suitable weak solutions to the Navier-Stokes equations in critical Morrey spaces, J. Math. Sc. 143 (2007), 2961--2968.
\bibitem[S2]{S2} Seregin, G.A.: Differential properties of weak solutions of the Navier-Stokes equations, Algebra and Analiz 14 (2002), 1--44 (in Russian).
\bibitem[S3]{S3} Seregin, G.A.: Local regularity of suitable weak solutions to the Navier-Stokes equations near the boundary.
\bibitem[S4]{S4} Seregin, G.A.: Lecture Notes on regularity theory for the Navier-Stokes equations, World Scientific Publishing Co. Pte. Ltd., Hackensack, NJ, 2015, x+258 pp. ISBN 978-981-4623-40-7.
\bibitem[SSS]{SSS} Seregin, G.A.; Shilkin, T.N.; Solonnikov, V.A.: Boundary partial regularity for the Navier-Stokes equations, Zap. Nauchn. Sem. POMI 310 (2004), 158--190.
\bibitem[SS1]{SS1} Seregin, G.A.; \v Sver\'ak, V.: Navier-Stokes equations with lower bounds on the pressure, ARMA 163 (2002), 65--86.
\bibitem[S]{S} Serrin, J.: On the interior regularity of weak solutions of the Navier-Stokes equations, ARMA 9 (1962), 187--195.
\bibitem[Z1]{Z1} Zaj\c{a}czkowski, W.M.: Global regular periodic solutions for equations of weakly compressible barotropic fluid motions.
\bibitem[Z2]{Z2} Zaj\c{a}czkowski, W.M.: Stability of two-dimensional solutions to the Navier-Stokes equations in cylindrical domains under Navier boundary conditions, JMAA 444 (2016), 275--297.
\bibitem[Z3]{Z3} Zaj\c{a}czkowski, W.M.: Global special regular solutions to the Navier-Stokes equations in a cylindrical domain under boundary slip conditions, Gakuto International Series, Mathematical Sciences and Applications Vol. 21 (2004), pp. 188.
\bibitem[Z4]{Z4} Zaj\c{a}czkowski, W.M.: Global special regular solutions to the Navier-Stokes equations in axially symmetric domains under boundary slip conditions, Diss. Math. 432 (2005), pp. 138.
\bibitem[Z5]{Z5} Zaj\c{a}czkowski, W.M.: global regular solutions to the Navier-Stokes equations in a cylinder, Banach Center Publications, Vol. 74 (2006), 235--255.
\bibitem[Z6]{Z6} Zaj\c{a}czkowski, W.M.: Global regular nonstationary flow for the Navier-Stokes equations in a cylindrical pipe, TMNA 26 (2005), 221--286.
\bibitem[Z7]{Z7} Zaj\c{a}czkowski, W.M.: On global regular solutions to the Navier-Stokes equations in cylindrical domains, TMNA 37 (2011), 55--85.
\bibitem[Z8]{Z8} Zaj\c{a}czkowski, W.M.: Some global regular solutions to Navier-Stokes equations, Math. Meth. Appl. Sc. 30 (2007), 123--151.
\end {thebibliography}

\end{document}